\documentclass[final]{amsart}
\usepackage{Einosty}
\usepackage[usenames]{color}
\usepackage{amssymb}
\newcommand{\hdarrow}{\stackrel{d_H}{\longrightarrow}}
\newcommand{\ghdarrow}{\stackrel{d_{GH}}{\longrightarrow}}
\DeclareMathOperator{\Tan}{Tan}
\newcommand{\MC}{\mathcal{M}_C}

\title{Locally rich compact sets}
\author{Changhao Chen}
\author{Eino Rossi}

\address{Changhao Chen\\
Department of Mathematical Sciences, P.O. Box 3000, 90014
University of Oulu, Finland}
\email{changhao.chen@oulu.fi}

\address{Eino Rossi\\
Department of Mathematics and Statistics, P.O. Box 35 (MaD) FI-40014 University of Jyv{\"a}skyl{\"a}, Finland}
\email{eino.rossi@jyu.fi}

\date{\today}
\thanks{Both authors were supported by the Vilho, Yrj{\"o}, and Kalle V{\"a}is{\"a}l{\"a} foundation.}
\begin{document}

\begin{abstract}
We construct a compact metric space that has any other compact metric space as a tangent at all points, with respect to the Gromov-Hausdorff distance. Furthermore, we give examples of compact sets in the Euclidean unit cube, that have almost any other compact set of the cube as a tangent at all points or just in a dense subset. Here the ``almost all compact sets'' means that the tangent collection contains a contracted image of any compact set of the cube and that the contraction ratios are uniformly bounded. In the Euclidean space, the distance of subsets is measured by the Hausdorff distance. Also the geometric properties and dimensions of such spaces and sets are studied.
\end{abstract}
\maketitle

\section{Introduction}
Tangent measures and sets  give information of the local structure of a Radon measure or a compact set. These tangent objects often have more regular structure than the original object in question. For example, a tangent set of a totally disconnected self-affine carpet can contain intervals, see \cite{BandtKaenmaki2013}. However, this is not always the case. O'Neil \cite{ONeil1995} constructed a Radon measure $\mu$ on $\R^d$ that has any other Radon measure of $\R^d$ as a tangent measure at $\mu$ almost all points. Furthermore, in his PhD thesis \cite{Oneil1994}, he showed that this is a typical property of Radon measures. This was later re-proved by Sahlsten \cite{Sahlsten2014} by using a different method. Buczolich studied the (micro-)tangent sets of continuous functions and obtained results in a very similar fashion. He proved that the graph of a typical continuous function on [0,1] has a graph of any other continuous function (on [-1,1] for which $f(0)=0$) as a tangent set at $(x,f(x))$ for Lebesgue 
almost 
all $x\in(0,1)$, see \cite[Theorem 5]{Buczolich2003}. This work was later extended to the local maximum and minimum points of $f$ by Buczolich and R\'{a}ti \cite{BuczolichRati2006}. The dynamics of the ``zooming in'' operation or ``scenery flow'' has also been studied. For example, \cite{BedfordFisher1996,BedfordFisher1997,BedfordFisherUrbanski2002} study the scenery flow of Cantor sets and Julia sets. Furstenberg \cite{Furstenberg2008} studied the distribution of tangent measures in Euclidean unit cube and this idea has been further developed in the study of typical tangent measure distributions, see \cite{Hochman2010,KaenmakiSahlstenShmerkin2013,KaenmakiSahlstenShmerkin2014}.

We study compact metric spaces and compact subsets of Euclidean spaces, and their tangent properties. In our results, there are no measures involved. We call a compact metric space \emph{locally rich}, if it has a ``large'' collection of tangent spaces at all points. A precise definition will be given in Definition \ref{richmetricdef}. Our main result is that there exist locally rich spaces. We also show that locally richness is a typical property. A property $P$ of points $x \in X$ is satisfied for \emph{typical} $x \in X$ if the complement of set 
\[
 \{x \in X : x \text{ satisfies } P\}
\] 
is of first category. A subset of a topological space $X$ is of first category, if it is a countable union of sets whose closure in $X$ has empty interior; otherwise it is called of second category.  

The distance of compact sets $K$ and $F$ in a metric space $(X,d_X)$ can be measured by the Hausdorff distance, defined by
\[
 d_H(K,F)=\inf\{\eps:F\subset K^\eps\text{ and }K\subset F^\eps\},
\]
where $K^\eps=\{x\in X:\dist(x,K)<\eps\}$. One could highlight the ambient space by denoting $d_H^X$, but in this work the space in which distances are measured is always clear from the context, so we only write $d_H$. We use the standard notations $\dimh, \adimm, \ydimm$, and $ \dima$ to denote the Hausdorff, lower Minkowski, upper Minkowski and Assouad dimension, respectively. For basic properties of these dimensions, we refer to \cite{Falconer1990,Mattila1995} and especially \cite{Luukkainen1998} for the Assouad dimension.

The paper is organized as follows.  In Section 1, we present the basic definitions. In Sections 2 and 3 we study locally rich sets in metric spaces and Euclidean spaces respectively. In Section 4, we construct an infinitely generated self-similar set and consider its tangent sets. In Section 5, we study the geometric properties of locally rich sets. To finish the paper, we consider other ways to zoom sets.

\subsection{Tangents of a metric space}
The distance of different metric spaces is measured by the Gromov-Hausdorff distance. For references, see \cite{AmbrosioTilli2004,Gromov1999,Heinonen2003} or \cite{HeinonenKoskelaShanmugalingamTyson2011}. The idea is that the distance of two separable metric spaces is measured by the infimum of the Hausdorff distances of their isometric images in $l^\infty$, the space of bounded real valued sequences. This is possible by the Fréchet embedding theorem\footnote{This theorem is sometimes represented as a consequence of the Kuratowski embedding theorem. We learned from \cite[Notes to chapter 3]{HeinonenKoskelaShanmugalingamTyson2011} that actually Fréchet \cite{Frechet1909} proved his theorem already in 1909, while Kuratowski's paper \cite{Kuratowski1935} appeared in 1935.} \cite[Theorem 3.1.11]{HeinonenKoskelaShanmugalingamTyson2011}. When only considering compact spaces, we could instead use embedding to $C([0,1])$, the space of continuous functions on $[0,1]$, by a theorem of Banach \cite[
Théorème 9, p. 185]{Banach1955}. 
This space has the advantage of being separable. However, $l^\infty$ is enough for us, since we only need the fact that $l^\infty$ is a normed space. The Gromov-Hausdorff distance of two separable metric spaces $X$ and $Y$ is defined as
\[
 d_{GH}(X,Y)=\inf_{i,j}d_H(i(X),j(Y)),
\]
where the Hausdorff distance is considered in $l^\infty$ and the infimum is taken over all isometric embeddings of $X$ and $Y$ to $l^\infty$. This distance does not make a difference between isometric spaces and can, on the other hand, give the value $\infty$ for unbounded spaces. Therefore it is reasonable to consider only the isometry classes of non-empty compact metric spaces. A rather standard notation for this collection is $\MC$. The following proposition about $(\MC,d_{GH})$ is the base of our work on metric spaces.

\begin{proposition}\label{metricprop}
The distance $d_{GH}$ is a metric on $\MC$. Furthermore, the space $(\MC,d_{GH})$ is complete and separable, and the countable dense set can be chosen to be the set of finite spaces, where all distances are rational.
\end{proposition}
The fact that $(\MC,d_{GH})$ is a metric space, is proved in \cite[Proposition 4.5.2]{AmbrosioTilli2004}, completeness in \cite[Theorem 4.5.6]{AmbrosioTilli2004}, and separability in \cite[Theorem 2.4]{Heinonen2003}. Since $\MC$ is complete, Baire's theorem tells that it is of second category, and we may talk about typicality.

Let $(X,d_X)$ be a compact metric space and $x\in X$. We define the zooming  map $T_{r,x}$ for all $x\in X$ and $r>0$ by
\[
 T_{x,r}(X,d_X)=(B(x,r), r^{-1}d|_{B(x,r)}),
\]
where $B(x,r)$ is the closed ball in the original metric of $X$. The metric is often clear from the context and then we may write $T_{x,r}(X)$ for short. We say that a compact metric space $Y$ is a tangent space of $X$ at $x$ if there exists a sequence $r_i\searrow 0$ so that
\[
 T_{x,r_i}(X)\ghdarrow Y
\]
as $i\to\infty$. We denote the collection of all tangent spaces of $X$ at $x$ by $\Tan(X,x)$. Observe that $\Tan(X,x)$ is a closed set in $\MC$. From the definition of $T_{x,r}$ it is clear that $\diam T_{x,r}(X)\leq 2$ for any $X$. On the other hand, this could be modified to any other constant by adding a constant scaling factor in front of $r^{-1}$ in the definition of $T_{x,r}$. Therefore it is reasonable to focus on some class of compact metric spaces with a uniformly bounded diameter. Since our interest is in the geometry of the spaces and not so much on the diameter, we just choose to focus on the spaces of diameter at most one. Let us denote $\MC(1)=\{K\in\MC:\diam K \leq 1\}$. This subset is closed, and so it is also complete. It again follows, from the well known theorem of Baire, that $\MC(1)$ is of second category. 

\begin{definition}
\label{richmetricdef}
A metric space $X\in\MC(1)$ is called \emph{locally rich} if it satisfies $\Tan(X,x)=\MC(1)$ for all $x\in X$.
\end{definition}

\begin{remark}
Our definition of a tangent space of a metric space is not the standard one, which is also called weak tangent. It's more common to use the pointed Gromov-Hausdorff convergence to define tangents of metric spaces. For the definition of pointed Gromov-Hausdorff convergence, see \cite{BugaroBugaroIvanov2001} or \cite{HeinonenKoskelaShanmugalingamTyson2011}. A pointed metric space is a triple $(X, d_X, x)$ where $(X, d_X)$ is a metric space and $x \in X$. A pointed metric space $(W, d_W, w)$ is called a (weak) tangent of $(X, d_X)$ if there exists a sequence points $(x_n)^{\infty}_{n=1} \subset X$ and a sequence $t_n \searrow 0$ such that the pointed metric spaces $(X, t_n^{-1}d_X, x_n)$ converge in the pointed Gromov-Hausdorff convergence to $(W, d_W, w)$. If $x_n =x$ for all $n \in \N$, then  $(W, d_W, w)$ is called a (weak) tangent of $(X, d_X)$ at point $x$. For more details we refer to \cite{BugaroBugaroIvanov2001,HeinonenKoskelaShanmugalingamTyson2011}, and for some applications to geometry, we refer to
\cite{MackayTyson2010,TysonWu2005}. Note that for compact metric spaces  the pointed Gromov-Hausdorff convergence is ''equivalent'' to the ordinary Gromov-Hausdorff  convergence, see \cite[Exercise 8.1.2]{BugaroBugaroIvanov2001} or \cite[Proposition 10.3.5]{HeinonenKoskelaShanmugalingamTyson2011}.

Our definition was motivated by the mini-set of a compact set of Furstenberg \cite{Furstenberg2008} and  many recent works about the local structure of sets and measures, see \cite{BandtKaenmaki2013,Buczolich2003, BuczolichRati2006, Hochman2010,
KaenmakiSahlstenShmerkin2013,KaenmakiSahlstenShmerkin2014}. Our definition of tangent space only considers the ``unite ball'' of $(X, t_n^{-1}d_X, x)$ for all $n\in \N$ with center $x$, since $T_{x,t_n}(X, d_X)$ is the unit ball of $(X, t_n^{-1}d_X, x)$ with center $x$. Both  tangent spaces reflect the local structure of $X$, but the weak tangent of a metric space is often an unbounded metric space. We show one connection of the two different ways of defining tangents by the following fact. 

Let $X$ be a compact metric space, $(W, d_W, w)$ be a weak tangent of $X$ at point $x$, and $(W, d_W, w)$ be a length space. For the definition of length space, see \cite[Definition 2.1.6]{BugaroBugaroIvanov2001}. We assume that $X_n:=(X, t_n^{-1}d_X, x)$ converges to $(W, d_W, w)$, in the pointed Gromov-Hausdorff convergence. The result of \cite[Exercise 8.1.3]{BugaroBugaroIvanov2001} implies that  
 \[
 B_{X_n}(x_n, r) \ghdarrow B_W(w,r)
  \text{~~for every r > 0}.
 \] 
Recall that 
$T_{x,t_n} (X) = B_{X_n}(x, 1)$. Thus we have that $B_W(w,1) \in \Tan(X,x)$.
\end{remark}

\subsection{Tangent sets in Euclidean spaces}
When considering compact subsets of Euclidean space, we could of course measure distances with the Gromov-Hausdorff distance. Instead of working in $l^\infty$, a more natural way would be to modify $d_{GH}$ so that the infimum is taken over isometries of the Euclidean space. This is sometimes called the Euclidean Gromov-Hausdorff distance, and denoted by $d_{EH}$. It is obvious that $d_{GH}\leq d_{EH}\leq d_H$, so results concerning tangents obtained for $d_H$ are  also valid for $d_{GH}$ and $d_{EH}$. Our aim is to obtain tangent results with respect to $d_H$, and when this is not possible we allow a small scaling and translation, to obtain the desired tangent results. With this in mind, we give following notations.

Let $Q=[-1,1]^d$ and $\KK$ be the space of all non-empty compact subsets of $Q$. Let $\KK$ be endowed with the Hausdorff metric. It is fairly easy to show that $(\KK,d_H)$ is a compact metric space. For a proof in a more general case, see \cite[2.10.21]{Federer1969}. For sub-collections $\A$ and $\mathcal{B}$  of $\KK$, we write $\A \lesssim  \mathcal{B}$, if there exists $\lambda_0 >0$, so that for every $A \in\A$, there exists a set $B \in \BB$, a real number $\lambda_0<\lambda<\lambda_0^{-1}$, a vector $a \in A$ and a vector $b \in B$, so that
\[
\lambda (A-a) = B - b.
\]
This relation is reflexive and transitive but not symmetric in general. If $\A\lesssim\BB$ and $\BB\lesssim\A$, then we write $\A\approx\BB$. The relation $\approx$ is an equivalence relation. Note that for any sub-collection $\A\subset\KK$, we trivially have $\A\lesssim\KK$ so to prove that $\A\approx\KK$ it suffices to show that $\KK\lesssim\A$. This fact is often used when we study tangent collections in Euclidean spaces.

Define mappings $\hat{T}_{x,t}\colon\R^d\to\R^d$, $\hat{T}_{x,t}(y)=\frac{y-x}{t}$ and $T_{x,t}\colon \KK\to\KK$,
\[
 T_{x,t}(E)=\hat{T}_{x,t}(E)\cap Q
\]
for all $x\in Q$ and $t>0$. We say that $F\in\KK$ is a tangent set of $E\in\KK$ at $x\in E$ if there exists a sequence $t_n\searrow 0$ so that
\[
 T_{x,t_n}(E)\hdarrow F.
\]
Since there is no danger of misunderstanding, we use same notation $\Tan(E,x)$ to denote the collection of all tangent sets of a set $E$ at $x$ as we did for the collection of all tangent spaces of a metric space $X$ at $x$. Due to compactness of $\KK$, this collection is never empty.

For $x\in\R^d$ and $t>0$, we often use the notation $Q(x,t)=[x-t,x+t]^d$. The Euclidean norm and the usual $\max$-norm of $x$ are denoted by $|x|$ and $\|x\|_{\max}$ respectively.

\begin{definition}
A compact set $K\in\KK$ is called \emph{locally rich at} $x$ if $\Tan(K,x)\approx\KK$. It is called \emph{locally rich} if $\Tan(K,x)\approx\KK$ for all $x\in K$.
\end{definition}

\subsection{Symbol space notations} The symbol space notations will often be used in our constructions. For the reader's convenience, we summarize them here. Let $\Sigma=\prod_{i=1}^\infty X_i$, where $X_i$ are fixed arbitrary sets.
An element $\omega$ of $\Sigma$ is a mapping $\omega\colon \N \to \bigcup_{i=1}^\infty X_i$ so that $\omega(n)\in X_n$ for all $n\in\N$. One can also think of $\omega\in\Sigma$ as a sequence $(\omega(1),\omega(2)\ldots)$. For any $n\in\N$ we denote $\Sigma_n=\prod_{i=1}^n X_i$. When $\omega\in\Sigma_n$ we denote $|\omega|=n$ and call this $n$ the length of $\omega$. If $\omega\in\Sigma$ then the length of $\omega$ is infinite. If the length of $\omega$ is greater than $n$, then let $\omega|_n$ denote the restriction of $\omega$ to length $n$. That is $\omega|_n\in\Sigma_n$ and $\omega|_n(m)=\omega(m)$, for all $m\leq n$. When $\omega\neq\omega'$, we let $\omega\wedge\omega'$ denote the symbol $\omega|_N$, where $N=\max_{n\in\N}\{n:\omega|_n=\omega'|_n\}$. When the length of $\omega$ is $n$, we denote $[\omega]=\{\omega'\in\Sigma:\omega=\omega'|_n\}$, and call it a cylinder. Note that the cylinder consists of sequences of the form
\[
 \omega\omega'=(\omega(1),\omega(2),\ldots,\omega(n),\omega'(1),\omega'(2),\ldots)
\]
for $\omega'$ of  infinite length.

\section{Locally rich metric spaces}\label{locricmet}
Our main result is that there exists a locally rich metric space. In Theorem \ref{metricmain} we construct an example of such a space. We also study the dimensional properties of this space. Finally, Theorem \ref{metrictypical} shows that a typical compact metric space is locally rich.

By Proposition \ref{metricprop}, we can choose a dense sequence  $(\gamma(n))_{n=1}^{\infty}$ in $\MC(1)$ so that each element is a finite space with rational distances. Set $\delta_n=\min\{d_{\gamma(n)}(x,y):x\neq y\}$ and fix a sequence $r_n\searrow 0$ so that $r_1=1$ and $r_n < \delta_{n-1}$ for all $n\geq 2$. Let $\rr(n)=\prod_{i=1}^n r_i$.

Heuristically, we want to scale the space $\gamma(2)$ and put it into a small neighborhood of each point of $\gamma(1)$ so that when zooming into any point of $\gamma(1)$ with an appropriate scale, one would approximately see $\gamma(2)$. Next we would scale $\gamma(3)$ into a smaller neighborhood of each point of the previous space so that an appropriate zoom gives a better approximation of $\gamma(3)$. It is possible to continue this way, but the lack of an ambient space, other than $l^{\infty}$, forces us to define new metric spaces at each step. On the other hand, metric spaces are just points and distances, so it is possible to just give the distances, say in a matrix, in these finite cases. Finally one could wish that this sequence converges to a compact metric space, and that the limit space is locally rich. However, we do not actually do such a construction, since the simplest way is to just give the final compact metric space and then prove that is has the desired properties.

Set $\Sigma=\prod_{n=1}^{\infty} \gamma(n)$. For $\omega,\omega'\in\Sigma$,  define
\begin{align*}
 d_{\Sigma}(\omega,\omega')
 &= \max_{n\in\N} \rr(n) d_{\gamma(n)}( \omega(n),\omega'(n) ).
\end{align*}
The structure of $\Sigma$ is illustrated in Figure \ref{sigmapicture}. Recall that $r_{n+1}< d_{\gamma(n)}(x,y)\leq 1$ for all different $x,y\in \gamma(n)$. It is straightforward to check that $(\Sigma,d_{\Sigma})$ is a metric space. Since $\gamma(n)\in\MC(1)$ for all $n\in\N$, we have $\diam(\Sigma)=\diam(\gamma(1))\leq 1$. Note also that since $r_{n+1}<\delta_n$, the definition of the metric is equivalent to $d_{\Sigma}(\omega,\omega') = \rr(n)d_{\gamma(n)}(\omega(n),\omega'(n))$, where $n=|\omega\wedge\omega'|+1$. Also, it is clear that $(\Sigma,d_{\Sigma})$ is compact. It is even possible to construct a converging sub-sequence by hand from a given sequence, since each element of the product is finite. On the other hand, one could argue that the metric $d_{\Sigma}$ gives the standard product topology and therefore $(\Sigma,d_{\Sigma})$ is compact by Tihonov's theorem. Now all that is left, is to show that $(\Sigma,d_{\Sigma})$ is locally rich. This is done in Theorem \ref{metricmain}, but first we 
give an approximation lemma. 

\begin{lemma}\label{inclusion}
 For all $\omega\in\Sigma$, we have $d_{GH}(T_{\omega,\rr(n)}(\Sigma),\gamma(n))\leq r_{n+1}$.
\end{lemma}
\begin{proof}
 For any $\omega\in\Sigma$ the space $T_{\omega,\rr(n)}(\Sigma)$ equals to $[\omega|_{n-1}]$ with the metric $\rr(n)^{-1}d_{\Sigma}$ (restricted to $[\omega|_{n-1}]$). Let $\alpha\in\prod_{i=n+1}^{\infty}\gamma(i)$. Then the set
 \[
  \{\omega|_{n-1} x \alpha:x\in\gamma(n)\}\subset T_{\omega,\rr(n)}(\Sigma)
 \]
 is isometric to $\gamma(n)$. The isometry is given by $x\mapsto\omega|_{n-1} x \alpha$, since
 \begin{align*}
  d_{T_{\omega,\rr(n)}(\Sigma)}(\omega|_{n-1} x \alpha,\omega|_{n-1} y \alpha)
  &= \rr(n)^{-1}d_{\Sigma}(\omega|_{n-1} x \alpha,\omega|_{n-1} y \alpha)
   =d_{\gamma(n)}(x,y).
 \end{align*}
 Letting $i$ denote this isometry in question, we clearly have
 \[
  d_H \big( T_{\omega,\rr(n)}(\Sigma), i(\gamma(n)) \big)\leq r_{n+1}
 \]
and so $d_{GH}(T_{\omega,\rr(n)}(\Sigma),\gamma(n))\leq r_{n+1}$.
\end{proof}

\begin{figure}
\centering
\includegraphics{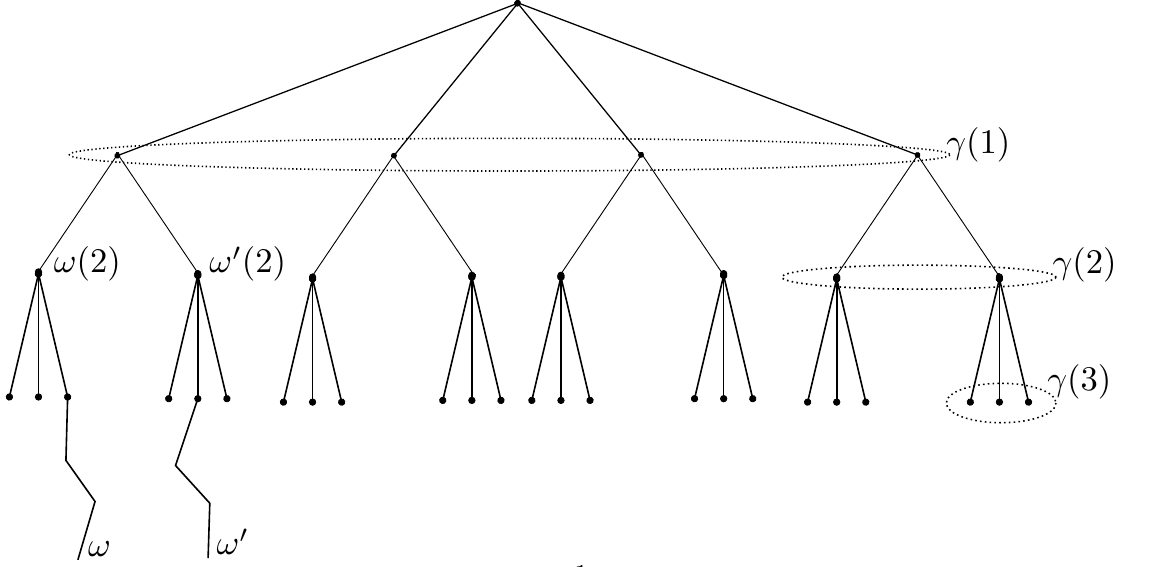}
\caption{The space $\Sigma$ is a tree, where the number of offspring in each vertex at level $n$ is $\#\gamma(n+1)$. Here the distance between $\omega$ and $\omega'$ equals to $\rr(2)d_{\gamma(2)}(\omega(2),\omega'(2))$. }
\label{sigmapicture}
\end{figure}

\begin{theorem}
\label{metricmain}
 For all $\omega\in\Sigma$, we have $\Tan(\Sigma,\omega)=\MC(1)$.
\end{theorem}
\begin{proof}
 Let $K\in\MC(1)$ and $\omega\in\Sigma$. For any $\eps$, we find an integer $n$, in fact infinitely many $n$, so that $d_{GH}(K,\gamma(n))\leq\eps$ and $r_{n+1}\leq\eps$. By Lemma \ref{inclusion}, we have
 \[
  d_{GH}(T_{\omega,\rr(n)}(\Sigma),K)
  \leq d_{GH}(T_{\omega,\rr(n)}(\Sigma),\gamma(n))+d_{GH}(K,\gamma(n))
  \leq 2\eps.
 \]
 Thus $K$ is a tangent space of $\Sigma$ at $\omega$. Thus we complete the proof by the arbitrary choice of $K$ and $\omega$.
\end{proof}

\begin{remark}
 In order to tangent out whole $\MC$ from $\Sigma$, instead of just $\MC(1)$, we could modify the definition of a tangent space by allowing a finite scaling of the metric at each step. This is similar to the definition of a tangent measure. On the other hand, we could concentrate on the similarity classes of metric spaces, meaning that two spaces are equivalent if they are the same up to isometry and scaling of the metric with a constant.
\end{remark}

In the next theorem, we show that a typical compact metric space is locally rich. It is interesting that this somewhat weird and irregular property is, in terms of category, very common, but this seems to be the nature of (Baire) typical objects.

\begin{theorem}
\label{metrictypical}
A typical point of $(\MC(1),d_{GH})$ is a locally rich space.
\end{theorem}
Before the proof we remark a few elementary properties of $d_H$. Let $A,B\subset l^\infty$ be compact and $a,b>0$. Then for all $x\in A$ and $y\in B$ we have $\| ax-by \| \leq a\|x-y\|+|a-b|\|y\|$, which implies that if $d_H(A,B)\leq \eps$ then
\begin{equation}
 \label{lobservation}
 d_H(aA,bB)\leq a\eps+|a-b|\max_{y\in B}\|y\|.
\end{equation}
Here $aA=\{ax:a\in A\}$. Also, note that if $i\colon X\to l^\infty$ is an isometric embedding, then for any $x\in X$ and $r>0$ we have that $i|_{B(x,r)}\colon B(x,r) \to i(B(x,r))$ is an isometry and so $i_{r}\colon T_{x,r}(X)\to r^{-1}i(B(x,r))$ defined by $i_{r}(y)=r^{-1}i(y)$ is an isometry.

\begin{proof}[Proof of Theorem \ref{metrictypical}]
``Similarly'' as before, we set $\gamma(n,k)=\gamma(n)\times\gamma(k)$ and endow this space with a metric
\[
 d_{n,k}(\alpha,\beta)=\max\{d_{\gamma(n)}(\alpha(1),\beta(1)),r_{n+1} d_{\gamma(k)}(\alpha(2),\beta(2))\},
\]
for $\alpha=(\alpha(1),\alpha(2))$ and $\beta=(\beta(1),\beta(2))\in \gamma(n,k)$. Recall that $r_{n+1}<\delta_n$.

Let $p\in \gamma(k)$,  then $\gamma(n)\times \{p\}$ with the metric $d_{n,k}$ is isometric to the metric space $(\gamma(n), d_{\gamma_n}).$ Let $x=(x_1,x_2) \in \gamma(n,k)$, then $d_{n,k}((x_1, x_2), (x_1,p))\leq r_{n+1}$. Thus we have 
\begin{equation}\label{2}
d_{GH}( \gamma(n,k), \gamma(n)\times \{p\}) \leq r_{n+1}.
\end{equation}
Since $\{\gamma(n)\}^{\infty}_{n=1}$ is a dense in $\MC(1)$, the estimate \eqref{2} gives that also $\{\gamma(n,k)\}^{\infty}_{n=1}$ is dense in $\MC(1)$. Note that the above argument holds for every $k\in \N$. Furthermore if $k \in \N$ and $N\in \N$, then the set $\{\gamma(n,k)\}^{\infty}_{n=N}$ is dense in $\MC(1)$. 

Fix a sequence $\eps_n\searrow 0$ so that $r_{n+1}+8 \eps_n<\delta_n$ and denote
\begin{equation}
\mathcal{G}(n,k):= U_{d_{GH}}(\gamma(n,k), r_{n+1}\epsilon_n),
\end{equation}
an open ball in $\MC(1)$ with center $\gamma(n,k)$ and radius $r_{n+1}\epsilon_n$. As a union of open balls, $\bigcup^{\infty}_{n=N} \mathcal{G}(n,k)$ is open for any $N\in\N$ and it is dense for any $N\in\N$, since $\{\gamma(n,k)\}^{\infty}_{n=N}$ is dense for any $N\in\N$. 

Fix $X\in\GG(k):= \bigcap^{\infty}_{N=1}\bigcup^{\infty}_{n=N} \mathcal{G}(n,k)$ and $x\in X$. Since $X\in\GG(k)$, there are infinitely many $n\in\N$ so that $d_{GH}(\gamma(n,k),X)<r_{n+1}\eps_n$. Thus we find isometries $i$ and $j$ so that $d_H(i(X),j(\gamma(n,k)))<r_{n+1}\eps_n$. This implies that for every $y\in X$ there exists $z\in \gamma(n,k)$ so that $d_{l^\infty}(i(y),j(z))<r_{n+1}\eps_n$ and vice versa. Especially, we find $z$ so that $d_{l^\infty}(i(x),j(z))<r_{n+1}\eps_n$. Fix $s_n=r_{n+1}+2r_{n+1} \eps_n$. Next we argue that
\begin{equation}
 \label{epsr}
 d_{H}\left( i(B(x,s_n)) , j(B(z,r_{n+1})) \right)\leq r_{n+1}\eps_n.
\end{equation}
Recall that $\delta_n>r_{n+1}+8r_{n+1}\eps_n$. By this and the triangle inequality we see that
\[
 (i(X)\cup j(\gamma(n,k))) \cap ( B(i(x),r_{n+1}+4r_{n+1}\eps_n)\setminus B(i(x),s_n) )=\emptyset.
\]
Here the balls $B(i(x),r)$ are (closed) balls of $l^\infty$. Since this annulus is empty, any point of $j(\gamma(n,k)))$ ``near'' to any point of $i(B(x,s_n))$ must lie inside $B(i(x),s_n)$. This combined with the fact that $d_H(i(X),j(\gamma(n,k)))<r_{n+1}\eps_n$ proves \eqref{epsr}.

Without loss of generality we may assume that $0\in j(B(z,r_{n+1}))$. By observation \eqref{lobservation} and equation \eqref{epsr}, we now have that
\begin{align*}
 d_{GH}( T_{x,s_n}(X) , T_{z,r_{n+1}}(\gamma(n,k)) )
 &\leq d_{H}( i_{ s_n }(B(x,s_n)) , j_{ r_{n+1} }(B(z,r_{n+1}) )\\
 &= d_{H}( s_n^{-1} i(B(x,s_n)) , r_{n+1}^{-1} j(B(z,r_{n+1})) )\\
 &\leq \frac{r_{n+1}\eps_n}{s_n} + ( \frac{1}{r_{n+1}} - \frac{1}{s_n} ) \diam j(B(z,r_{n+1}))  \\
 &\leq \frac{\eps_n}{1+2\eps_n} + \frac{2\eps_n}{r_{n+1}(1+2\eps_n)} 2r_{n+1},
\end{align*}
which converges to $0$ as $n\to \infty$ since $\eps_n \to 0$ as $n\to \infty$. It remains to be proved that $T_{z,r_{n+1}}(\gamma(n,k))$ is isometric to $\gamma(k)$, but this follows by a similar argument as the proof of Lemma 2.1.

We have obtained that $\gamma(k)\in\Tan(X,x)$ for all $X\in\GG(k)$ and $x\in X$. Thus from every element of $\bigcap^{\infty}_{k=1} \GG(k)$ one can tangent out any $\gamma(k)$ at all points. As we observed in the introduction, $\Tan(X,x)$ is closed, so any $X\in\bigcap^{\infty}_{k=1} \mathcal{G}(k)$ is locally rich.

In the end, we notice that the complement of $\bigcap^{\infty}_{k=1} \mathcal{G}(k)$ is of first category since $\mathcal{G}(n,k)$ are open and dense sets in a complete metric space $\MC(1)$.
\end{proof}

\subsection{Properties of $\Sigma$}
A metric space $X$ is called \emph{locally doubling} if there exists a constant $N(x)\in\N$, so that every ball $B(x,r)\subset X$ can be covered by at most $N(x)$ balls of radius $r/2$, where $N(x)$ does not depend on $r$. If one can choose a uniform $N$ for all $x\in X$, then $X$ is called $\emph{doubling}$. As one might guess, $(\Sigma,d_{\Sigma})$ is not locally doubling anywhere. Intuitively the reason is that in small scales, $\Sigma$ contains approximations of spaces with larger and larger doubling constants.

\begin{remark}
 \label{doublingremark}
 A locally rich space is not locally doubling anywhere, and thus has Assouad dimension equal to $\infty$.
\end{remark}
\begin{proof}
 The latter claim follows by from the first one, see \cite[Proposition 1.15]{Heinonen2003}. Assume on the contrary that a locally rich space $(X,d_X)$ is locally doubling at $x$ with constant $N$. Let $K$ be a metric space of $N + 1$ elements with all distances equal to $1$. Since $K\in\Tan(X,x)$, we find sequence $t_n\searrow 0$, for which $\# T_{x,t_n}(X)\geq N+1$ and a set $A\subset T_{x,t_n}(X)$ with $\# A = N+1$ and $d_{T_{x,t_n}(X)}(x,y)>3/4$ for all different $x,y\in A$. Therefore, it takes at least $N+1$ balls of radius $2^{-1} t_n$ to cover the ball $B(x,t_n)$.
\end{proof}

\begin{remark}
\label{minkowskiremark}
 From Remark \ref{doublingremark} one can see that $\dima\Sigma=\infty$ regardless of the choice of the sequence $(r_n)$ in the construction. It turns out that the other dimensions depend heavily on $(r_n)$. For example, if $r_n\leq (\#\gamma(n))^{-n}$, we have that
 \begin{align*}
  \frac{\log N(\Sigma,\rr(n))}{-\log \rr(n)}
  &\leq \frac{\log \prod_{i=1}^{n}\#\gamma(i)}{\log\prod_{i=1}^{n}(\#\gamma(i))^i}
  \to 0,
 \end{align*}
where  $N(\Sigma,\rr(n))$ denotes the minimal number of balls with radius $\rr(n)$ needed to cover $\Sigma$. This  implies that $\adimm \Sigma=0$.
 
We do not know how to estimate the Hausdorff dimension of $\Sigma$ from below. One way could be to consider the lower local dimension of a Borel measure on $\Sigma$. For definitions and properties of local dimensions of measures, see \cite{Falconer1997}. Let $\mu$ be the unique measure satisfying $\mu[\omega|_n]=\prod_{i=1}^n p_i(\omega(i))$ where $p_i(x)$ are positive real numbers for each $x \in \gamma(i)$
satisfying $\sum_{x \in \gamma(i)}p_i(x)=1$ for all $i \in \N$. Choosing $\rr(n+1)<r \leq \rr(n)$ gives that
 \begin{align*}
  \frac{\log \mu B(\omega,r)}{\log r}
  &\geq \frac{\log \mu B(\omega,\rr(n))}{\log \rr(n+1)}
   = \frac{\log \mu [\omega|_n]}{\log \rr(n+1)}
   = \frac{\log \prod_{i=1}^n p_i(\omega(i))}{\log \rr(n+1)}.
 \end{align*}
Choosing the weights $p_i(x)$ and the sequence $(r_n)$ fixes the estimate for the lower local dimension of $\mu$ for all $\omega$ and therefore also for $\dimh \Sigma$. However, it is not clear how to choose $(r_n)$ large enough. The conditions are that $r_n\searrow 0$ and $r_n < \delta_{n-1}$, the minimum of distances between points of $\gamma(n-1)$ and it takes a lot more detailed construction to get a grip of $\delta_n$. On the other hand in Theorem \ref{mattilacorollary}, we show that a set $E$ in Euclidean space, might have positive Hausdorff dimension, even if one can tangent out $\{0\}$ from each $x\in E$. Therefore we conjecture that a locally rich space can have positive Hausdorff dimension.
\end{remark}

\section{Locally rich sets in Euclidean spaces}\label{Euclidean}
In this section, we show that there exists locally rich sets in $Q$. Theorem \ref{allpoints} gives an example of such a set, and in Theorem \ref{Baire} we show that this is a typical property of compact sets. The construction is of course similar to the construction of the locally rich metric space, but the geometry of $\R^d$ and the definition of $d_H$ give some technical differences. For example, there needs to be enough space between the construction pieces. Also, in any compact set $K$, there are always points $x$, for which $T_{x,r}(K)$ is at least half empty for all $r>0$.

Let us fix the notations
\begin{align*}
 \KK_0 &=\{K\subset Q: K \text{ compact and }0\in K\}\\
 \mathcal{S}&=\{\bigcup_{i=0}^{N-1}x_i:
 {x_i}\in \Q^d\cap (-1,1)^d,x_0=0,N<\infty\}.
\end{align*}
Clearly $\mathcal{S}$ is countable and dense in $\KK_0$. Choose a sequence $(\gamma_i)_{i=1}^{\infty} \subset \mathcal{S}$, so that each element of $\mathcal{S}$ occurs infinitely many times. Note that each $\gamma_i$ is finite. Let 
\begin{equation}\label{delta}
\delta_n=\min\{|x-y|: x\in\gamma_n, y \in \gamma_n \cup \partial Q ~~\text{and} ~~x\neq y\}.  
\end{equation}
Fix a sequence $(r_n)$ with $r_1=1$, $r_i\searrow 0$ and let $\rr(n)=\prod_{i=1}^n r_i$.

\begin{theorem}
\label{Oneiltheorem}
There exists a compact set $X$ satisfying $\Tan(X,x)=\KK_0$ for all $x$ in a dense subset of $X$.
\end{theorem}

Notice that  for any compact set $E \in \KK$, we have  $\Tan(E,x)\neq\KK_0$ at the ``boundary'' point of $E$. Here the ``boundary''' of $E$ is the set $E\cap \partial Q'$ where $Q'$ is the smallest cube that contain $E$ and $\partial Q'$ is the boundary of $Q'$. Since every element of $\Tan(X,x)$ contains the original point zero, the key in obtaining Theorem \ref{allpoints} is that we do not try to get  $\Tan(X,x)=\KK$ but only $\Tan(X,x)\approx\KK$.

\begin{theorem}
\label{allpoints}
 There exists a compact set $X$ satisfying $\Tan(X,x)\approx\KK$ for all $x\in X$.
\end{theorem}
Theorem \ref{Oneiltheorem} is similar to \cite[Theorem 3]{ONeil1995}, but does not follow directly from that, since generally the convergence of measures does not imply the convergence of the supports of the measures. Not even in the case when the supports are finite, see Example \ref{Toby}. We may nevertheless use the construction of \cite{ONeil1995}, we just need to check the convergence of compact sets instead of the weak convergence of measures. Note that the same construction is used to prove Theorems  \ref{Oneiltheorem} and \ref{allpoints}. It is rather surprising that this tangent property holds for all points, since it not possible for measures, see \cite[Proposition 5.1]{Sahlsten2014}. 

Let us now construct the compact set in question in Theorems \ref{Oneiltheorem} and \ref{allpoints}.  Unlike in the construction of a locally rich metric space, we require that 
\begin{equation}\label{ddistance}
8\sqrt{d}r_{n+1}\leq \delta_n,
\end{equation}
to guarantee that we have enough empty space in between the screens that we are focusing on. Here $d$ is the dimension of the ambient space. When $\ii\in\Sigma=\prod_{i=1}^\infty \gamma_i$, let $\ii(n)$ denote the $n$:th coordinate of $\ii$. Define the projection $\pi\colon\Sigma \to Q$ by setting 
\[
\{\pi\ii\}= \bigcap_{n=1}^{\infty}\left( \rr(n+1)Q+\sum^{n}_{k=1} \rr(k)\ii(k) \right).
\]
We get that $\pi$ is well defined because the cubes are nested, meaning that $\rr(n+1)Q+\sum_{k=1}^n \rr(k)\ii(k) \subset \rr(n)Q+\sum_{k=1}^{n-1} \rr(k)\ii(k)$: If $y\in \rr(n+1)Q+\sum_{k=1}^n \rr(k)\ii(k)$, then
\begin{align*}
 \|\sum_{k=1}^{n-1} \rr(k)\ii(k)-y\|_{\max}
 &\leq \|\sum_{k=1}^{n-1} \rr(k)\ii(k)-\sum_{k=1}^{n} \rr(k)\ii(k)\|_{\max} \\
 &\ \ \ + \|\sum_{k=1}^{n} \rr(k)\ii(k)-y\|_{\max}\\
 &=\|\rr(n)\ii(n)\|_{\max}+\rr(n+1)\\
 &\leq \rr(n)(1-\delta_n)+\rr(n+1)\\
 &\leq \rr(n)(1-r_{n+1})+\rr(n+1)\\
 &\leq \rr(n),
\end{align*}
which proves that $\rr(n+1)Q+\sum_{k=1}^n \rr(k)\ii(k) \subset \rr(n)Q+\sum_{k=1}^{n-1} \rr(k)\ii(k)$. We also have that $\pi$ is a continuous injection: Let $\ii,\jj\in\Sigma$ and $\ii\neq\jj$. Write $n_0=\min\{n:\ii(n)\neq\jj(n)\}$. From the previous calculation, it follows that
\begin{equation}
 \label{picontinuity}
 \|\pi\ii-\pi\jj\|_{\max} \leq 2\rr(n_0-1),
\end{equation}
which gives the continuity (with respect to the standard tree metric on $\Sigma$). Also, note that
\[
 \sum_{k=n_0}^\infty \rr(k) \leq \rr(n_0)\sum_{k=0}^\infty r_{n+1}^k
 = \rr(n_0 ) \frac{1}{1-r_{n+1}} \leq 2 \rr(n_0),
\]
when $r_{n+1}<1/2$, which is true for all $n$ by \eqref{ddistance}. Using this geometric argument, equation \eqref{ddistance} and the reverse triangle inequality, we get
\begin{equation}
 \label{piinjection}
 \begin{split}
  |\pi\ii-\pi\jj|
 &= |\sum_{k=n_0}^\infty \rr(k)( \ii(k)-\jj(k) )| \\
 &\geq |\rr(n_0)\left( \ii(n_0)-\jj(n_0) \right)|-|2\sqrt{d}\sum_{k=n_0 +1}^\infty \rr(k)|\\
 &\geq \rr(n_0)\delta_{n_0}-2\sqrt{d}2 \rr(n_0 +1)\\
 &= \rr(n_0)(8\sqrt{d}r_{n+1}-4\sqrt{d} r_{n+1})\\
 &= \rr(n_0 +1)4\sqrt{d},
 \end{split}
\end{equation}
which ensures the injectivity. Furthermore, $|\pi\ii-\pi\jj|>2\sqrt{d}\rr(n+1)$, which is needed in the proof of Theorem 3.2.

Finally, let $\mu$ be a Bernoulli measure on $\Sigma$ (e.g. $\mu[\ii]=\prod_{n=1}^{|\ii|}(\#\gamma(n))^{-1}$). For each $E\in \mathcal{S}$ fix a sequence $e(n)$ so that $\gamma_{e(n)}=E$ and define
\[
 V_E=\{\ii\in\Sigma: \ii(e(n))=0 \text{ infinitely often}\}.
\]
A Borel-Cantelli type argument \cite[Lemma 4]{ONeil1995} shows that $\mu(\bigcap_{E\in\mathcal{S}}V_E)=1$. The set $\pi\Sigma$ is compact since $\Sigma$ is compact and $\pi$ is continuous.
\begin{lemma}
\label{Oneillemma}
 If $\ii\in\bigcap_{E\in\mathcal{S}}V_E$, then $\mathcal{S}\subset \Tan(\pi\Sigma,\pi\ii)$.
\end{lemma}
\begin{proof}
 Let $E\in\mathcal{S}$ and $\ii\in \bigcap_{E\in\mathcal{S}}V_E$. Let $k(n)$ be a sub-sequence of $e(n)$ so that $\ii(k(n))=0$ for all $n\in\N$.
 Let $n$ be fixed. For $y\in E$ let $\jj_y$ be so that $\ii(m)=\jj_y(m)$ for all $m\neq k(n)$ and $\jj_y(k(n))=y$. Recall that $\ii(k(n))=0$. We have that $\|\pi\ii-\pi\jj_y\|_{\max}=\rr(k(n))\|\pi\jj\|_{\max}\leq \rr(k(n))$ and so
 \[
 y=\frac{\pi\jj_y-\pi\ii}{\rr(k(n))}\in T_{\pi\ii,\rr(k(n))}(\pi\Sigma).
 \]
 Thus we conclude that $E\subset T_{\pi\ii,\rr(k(n))}(\pi\Sigma)$. On the other hand, if $\|\pi\mathtt{k}-\pi\ii\|_{\max}\leq \rr(k(n))$, then $\ii|_{k(n)-1}=\mathtt{k}|_{k(n)-1}$ by \eqref{piinjection}. Thus there exists $y\in E$ so that $\mathtt{k}(k(n))=y$ and so $\|\pi\jj_y-\pi\mathtt{k}\|_{\max}\leq 2\rr(n+1)$ by \eqref{picontinuity}. Combining these geometric estimates we obtain that
\[
 d_H(T_{{\pi\ii},\rr(k(n))}(\pi\Sigma),E)\leq 2 \sqrt{d} r_{k(n)+1}.
\]
This proves the claim, since $r_n\to 0$ as $n\to\infty$.
\end{proof}
\begin{proof}[Proof of Theorem \ref{Oneiltheorem}]
 Set $X=\pi\Sigma$ and  $m=\pi\mu$. By definitions of tangent sets and $\KK_0$, it is obvious that $\Tan(X,x)\subset\KK_0$ for all $x\in X$. By combining Lemma \ref{Oneillemma} and the fact that tangent collections are closed, we get that $\Tan(X,x)\supset\KK_0$ for $m$ almost all $x$. It is clear that any set $A$ with $m(A)=1$ is dense in $X$.
\end{proof}

\begin{proof}[Proof of Theorem \ref{allpoints}]
 It suffices to find $X$, for which $\KK_0\lesssim\Tan(X,x)$ for all $x\in X$, since $\KK_0\approx \KK$. Set $X=\pi\Sigma$. Let $E\in\KK$ and $x=\pi\ii$. Fix a sequence $k(n)$ so that $\gamma_{k(n)}\to E$ in the Hausdorff metric. Denote $\ii(n)=i_n$. Since the origin is in every $\gamma(n)$, we can define $\jj_n=(i_1,\dots,i_{n-1},0,0,\ldots)$. Consider the point $T_{\pi\jj_{k(n)},\rr_{k(n)}}(x)=:x_n$, see Figure \ref{zoompicture}. This point is  in $Q$ so the sequence $(x_n)$ has a converging sub-sequence. Let $j(n)$ be so that $x_{k(j(n))}$ converges to $\overline{x}\in Q$. By the choice of $k(n)$,
 \begin{equation*}
  T_{\pi\jj_{k(j(n))},\rr_{k(j(n))}}(X)\hdarrow E.
 \end{equation*}
 By this and the assumption \eqref{ddistance}, we have that
 \begin{equation*}
  T_{\pi\ii,2\rr_{k(j(n))}}(X)
  \hdarrow
  aE+b,
 \end{equation*}
 where $a=\frac{1}{2}$ and $b\in Q$ is so that $aE+b\subset Q$.
\end{proof}

\begin{figure}
\centering
\includegraphics{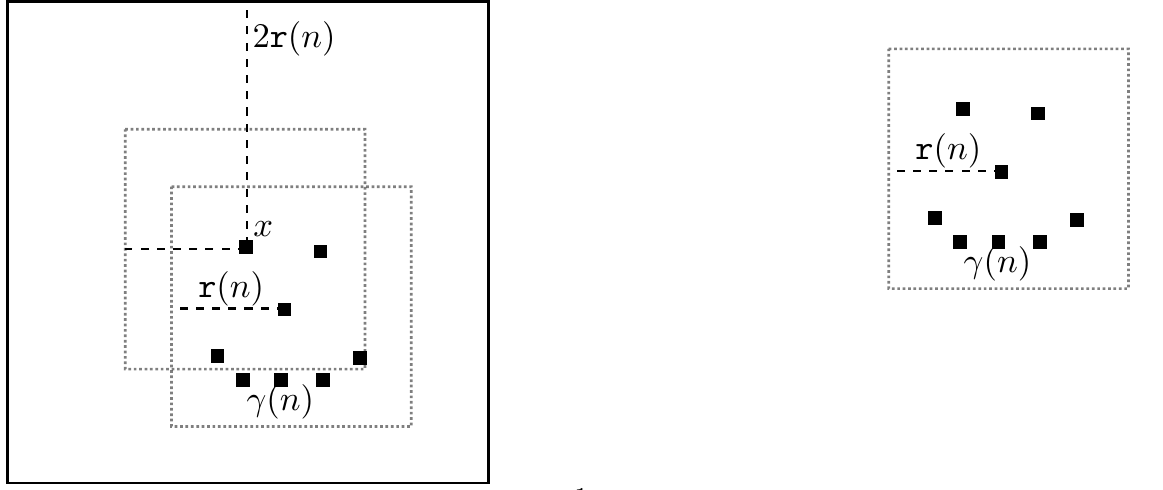}
\caption{The point $x=\pi\ii$ is not in the ``middle of $\gamma(n)$'', so one must use a doubled screen to see the whole set $\gamma(n)$. The assumption \eqref{ddistance} ensures that the doubled screen does not contain points $\pi\jj$ with $\jj(n-1)\neq \ii(n-1)$.
}
\label{zoompicture}
\end{figure}

\begin{example}
\label{Toby}
 Let $\delta_x$ denote the Dirac measure at $x$. For all $n\in\N$, let $\mu_n=(1-n^{-1})\delta_0+n^{-1}\delta_x$, where $x\neq 0$. Then $\mu_n\to\delta_0$ and so $\spt\mu=\{0\}$ but $\spt\mu_n=\{0,x\}$ for all $n$.
\end{example}
\begin{remark}\label{zero}
 As in remark \ref{minkowskiremark}, we can  modify this construction, namely by modifying the sequence $r_i$, so that the set $\pi\Sigma$ has Hausdorff dimension zero. By a similar technique we can also adjust the construction of O'Neil, such that $\dim_H \mu =0$, where $\mu$ is the measure of \cite[Theorem 3]{ONeil1995}. Here $\dim_H \mu$ is the Hausdorff dimension of $\mu$ as defined in \cite[(10.8)]{Falconer1997}. This is sometimes called the lower Hausdorff dimension of $\mu$.
\end{remark}

\begin{theorem}\label{Baire}
A typical compact set of $\KK$ is locally rich.
\end{theorem}
For each $n\in \N$, decompose the cube $Q$ into $3^{nd}$  disjoint $2\cdot 3^{-n}$-adic subcubes (note that $Q=[-1,1]^{d}$). Let $\mathcal{D}_n$ be the collection of these $3^{nd}$ subcubes. Define $C_n := \cup_{C \in \mathcal{D}_n} \{x_c\}$ where $x_c$ is the center point of the subcube $C$. Recall the definition of $\gamma_n$ and $\delta_n$ at the beginning of this section. Let $A_n \subset C_n$ and define $A_{n,k}:= \cup_{a\in A_n}(a+3^{-(n+1)}\gamma_k)$. It is easy to see that $A_{n,k}\in \KK$. Let 
\begin{equation}
\GG(n,k):= \bigcup_{A_n \subset C_n}U_{\KK}(A_{n,k},3^{-2n}\delta_k),
\end{equation}
where $U_{\KK}(A_{n,k}, 3^{-2n}\delta_k)$  is the open ball of $\KK$ with center $A_{n,k}$ and radius $3^{-2n}\delta_k$.  

\begin{lemma}\label{b}
Let $E\in \GG(n,k)$, then for each $x\in E$, there exist vector $b(x)\in Q(0,1/2)$ such that 
\begin{equation}\label{11}
d_H( T_{x,2\cdot3^{-(n+1)}}(E), \frac{1}{2}\gamma_k+b(x))<3^{-n}\delta_k.
\end{equation}
\end{lemma}
\begin{proof}
Since $E\in \GG(n,k)$, there exist $A_{n,k} \subset C_n$ so that $E \in U_{\KK}(A_{n,k},3^{-2n}\delta_k)$. By an elementary geometric argument, we obtain that for each $x\in E$, there exists a unique point $a\in A_{n,k}$ such that
\begin{equation}\label{22}
E\cap Q(a, 3^{-(n+1)}) = E\cap Q(x, 2\cdot3^{-(n+1)})
\end{equation}
and  
\begin{equation}\label{33}
d_H(E\cap Q(a, 3^{-(n+1)}), a+3^{-(n+1)}\gamma_k ) \leq 3^{-2n}\delta_k.
\end{equation}
Note that  there exists $b(x) \in Q$ such that
\begin{equation}\label{m}
T_{x, 2\cdot3^{-(n+1) }}(Q(a, 3^{-(n+1)}))= \frac{1}{2}Q +b(x).
\end{equation}
Since  $(\frac{1}{2}Q +b(x)) \subset Q$, we have $b(x)\in Q(0, 1/2)$. In the end by applying  estimates \eqref{22}, \eqref{33} and \eqref{m}, we arrive the estimate \eqref{11}.
\end{proof}

\begin{proof}[Proof of Theorem \ref{Baire}]
For any $k$ and $N$, the set $\bigcup_{n=N}^{\infty}\GG(n,k)$ is open. It is also dense, since any compact set can be approximated by $2\cdot3^{-n}$-adic cubes. Let
\[
 \GG_k: =\bigcap_{N=1}^{\infty} \bigcup_{n=N}^{\infty} \GG(n,k),
\]
which is a countable intersection of dense open sets. Let $E \in \GG_k$, then there exists a sequence $(n_i)^{\infty}_{i=1}$ ($n_i\nearrow \infty$ as $i\rightarrow\infty$) such that $E \in \GG(n_{i},k)$ for $i \in \N.$  Lemma \ref{b} implies that for very $x \in E$ and above $n_i$, there is $b_{n_i}(x) \in Q(0,1/2)$ such that the estimate \eqref{11} holds. Since $(b_{n_i}(x))^{\infty}_{i=1}\subset Q(0,1/2)$, we can find a sub-sequence $(b_{j}(x))$ of $(b_{n_i}(x))^{\infty}_{i=1}$ such that  $b_{j}(x) \rightarrow c(x)$ where $c(x) \in Q(0,1/2)$.  Thus we have $\{\frac{1}{2}\gamma_k + c(x) \} \in \Tan(E,x)$.

By the above discussion,  we have that $\bigcap_{k=1}^{\infty}\GG_k$ is a countable intersection of open and dense sets, and thus its complement is of first category. For each $E \in \bigcap_{k=1}^{\infty}\GG_k$, we have  $\Tan(E,x) \approx \KK_0\approx \KK$ for all $x\in E$. It means that every element of $\bigcap_{k=1}^{\infty}\GG_k$ is a locally rich set, thus we have finished the proof.
\end{proof}

\begin{remark} Define $\widetilde{\GG}(n,k)$ by
\begin{equation}
\widetilde{\GG}(n,k):= \bigcup_{A \subset C_n}U_{\KK}(A ,(\# \gamma_k)^{-n}3^{-n^{2}}),
\end{equation}
where $(\# \gamma_k)$ means the cardinality of $\gamma_k$. Let 
\[
\widetilde{\GG}:=  \bigcap^{\infty}_{k=1} \bigcap^\infty_{N=1} \bigcup^{\infty}_{n=N}\widetilde{\GG}(n,k).
\]
By the same argument as above, we have that the complement of $\widetilde{\GG}$ is of first category. Notice that for every $E \in \widetilde{\GG}(n,k)$, we can find at most  $(\# \gamma_k)3^n$ balls with radius  $(\# \gamma_k)^{-n}3^{-n^{2}}$ so that they cover $E$. It implies that every element $E$ of $\widetilde{\GG}$ satisfies $\underline{\dim}_M E=0$. Thus a typical compact set of $\KK$ has zero lower box-counting dimension. The result that typically a compact has zero lower box-counting dimension, was proved earlier in \cite{FengWu1997}, by a different method.
\end{remark}

\section{Locally rich infinitely generated self-similar set}
Here we consider the sub-space $\KK_0^+$ of $\KK$, where
\[
 \KK_0^+=\{K\in\KK:0\in K \text{ and }  x_1\geq 0 \text{ for all } x\in K\}.
\]
It is clear that $\KK_0^+\approx\KK$ and $\KK_0^+$ is also a separable metric metric space. As in the previous constructions, we choose a countable dense set $\{\gamma_n\}_{n=1}^\infty$ in $\KK_0^+$. Moreover we require that each set $\gamma_n$ is a finite union points and $0\in\gamma_n$. 
Fix sequences $a_n \searrow 0$ and $\lambda_n \searrow 0$ with the following properties
\begin{enumerate}
 \item $a_1+\lambda_1\leq 1$
 \item $a_{n+1}+\lambda_{n+1}<\frac{a_{n}}{n} $ for all $n\in\N$
 \item $\frac{a_n}{\lambda_n}\to 0$ as $n\to\infty$
\end{enumerate}
For example, one could choose $a_n=2^{-n^2}$ and $\lambda_n=n a_n$ for all $n\in\N$. By setting $\widetilde{\gamma}_n= \vec{a}_n +\lambda_n\gamma_n$, where $\vec{a}_n=(a_n,0,0,\ldots)$, we have $\widetilde{\gamma}_n \subset Q$ for all $n\in\N$. Set 
\[
C_0=\{0\}\cup\bigcup_{n=1}^\infty \widetilde{\gamma}_n.
\]
The set $C_0$ is made of the elements of $(\gamma_n)_{n=1}^{\infty}$ scaled with $\lambda_n$ and positioned so that the leftmost point of the $n$:th piece on the $x_1$-axis is $a_n$. The condition $(1)$ ensures that the construction stays in $Q$ and $(2)$ gives the separation of the pieces. Obviously zero is the only clustering point of $C_0$, so $C_0$ is compact. The next theorem shows that the tangent space of $C_0$ at zero is large, as might be expected.

\begin{figure}
\centering
\includegraphics{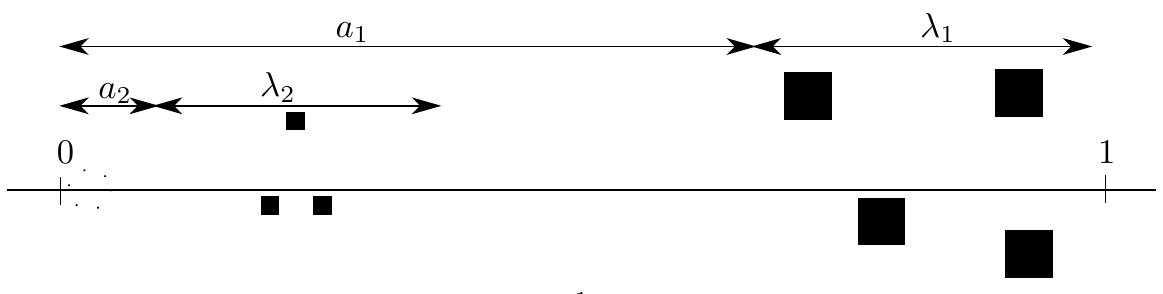}
\caption{First level approximation of the set $C_\infty$}
\label{ifspicture}
\end{figure}

\begin{theorem}
 \label{thmC00}
 We have  $\Tan(C_0,0)= \KK_0^+$.
\end{theorem}
\begin{proof}
It is obvious that $\Tan(C_0,0)\subset \KK_0^+$. For any $K\in\KK_0^+$, there exist a sub-sequence $\{\gamma_{n(k)}\}^{\infty}_{k=1}$, such that $\gamma_{n(k)} \hdarrow K$ as $k\rightarrow\infty$. By choosing real numbers $t_{n(k)}=a_{n(k)}+\lambda_{n(k)}$, we get that $T_{0,t_{n(k)}}(C_0) = A_{n(k)}\cup B_{n(k)}$, where
\[
 A_{n(k)}=\frac{1}{t_{n(k)}} \bigg(\{0\}\cup\bigcup_{m=n(k)+1}^\infty \widetilde{\gamma}_m \bigg)
\quad\text{and}\quad
B_{n(k)}= \frac{1}{t_{n(k)}}(\vec{a}_{n(k)} +\lambda_{n(k)} \gamma_{n(k)}).
\]
By the  triangle inequality, we have
\begin{align*}
 d_H(A_{n(k)} \cup B_{n(k)}, K)
 &\leq d_H (A_{n(k)} \cup B_{n(k)}, B_{n(k)}) +d_H(B_{n(k)}, K)\\
 &\leq \frac{d a_{n(k)}}{t_{n(k)}}+d_H(B_{n(k)}, K).
\end{align*}
By the choice of the sequences $(a_n)$ and $(\lambda_n)$, we have $a_{n(k)}/t_{n(k)}\to 0$ and $\lambda_{n(k)}/t_{n(k)}\to 1$ and so
$T_{0,t_{n(k)}}(C_0) \hdarrow K$. Thus the result follows.
\end{proof}

\subsection{Infinitely generated self-similar set}
Since each  $\widetilde{\gamma}_n$ is a finite set of points we may consider a countable set of similitude contractions, so that each point of $\widetilde{\gamma}_n$ is the image of zero under one of these contractions. More precisely, denote $\widetilde{\delta}_n=(8\sqrt{d})^{-1}\min\{|x-y|:x,y\in\gamma_n,x\neq y\}$ and set
\begin{equation}\label{ifs}
f_{n,m}(x)=\lambda_n \widetilde{\delta}_{n} \eps_n x+\xi_{n,m}
\end{equation}
for  $n\geq 1, n\in\N$ and $1\leq m \leq \#\gamma_n$, where $\eps_1\leq 1/2$ and $\eps_n\searrow0$ and the translations $\xi_{n,m}$ are fixed so that the union of $f_{n,m}(0)$ over $m$ equals $\widetilde{\gamma}_n$. We add the mapping $f_{0,1}\equiv 0$. Note that for fixed $n$, the mappings $f_{n,m}$ have the same contraction ratio $r_{(n,m)}$. Usually, we denote this number by $r_n$ for short ($r_{(0,1)} =0$).

For any $A\in\KK$, we define
\begin{equation}
 \label{phicompact}
 \phi(A) = \bigcup_{n=0}^\infty\bigcup_{m=1}^{\#\gamma_n}f_{n,m}(A)
\end{equation}
Let $(x_n)$ be a sequence in $\phi_k(A)$. If $(x_n)$ stays in a finite union of $f_{n,m}(A)$ where $n\geq 0$ and $m\geq 1$, then it has a converging sub-sequence since $f_{n,m}(A)$ are compact. If not, then it must have a sub-sequence that converges to $0\in\phi(A)$. Thus $\phi(A)$ compact. Therefore by setting $C_k = \phi^k(Q)$ for all $k\geq 1$, we get that $C_k\subset C_{k-1}$ for all $k\geq 2$ and the set $C_\infty$ defined by
\[
 C_\infty=\bigcap_{k=1}^\infty C_k,
\]
is compact and non-empty. The sequence $(\eps_n)$ ensures that $\Tan(C_1,0)=\Tan(C_0,0)$ and thus $\Tan(K,0)=\Tan(C_0,0)$ for any $C_0\subset K\subset C_1$. Especially this applies to $C_\infty$.

We set
\begin{equation}\label{symbol}
 I=\{(n,m): n \geq 1,1\leq m\leq \#\gamma_n\}
\end{equation}
and consider the symbol space $\IN$. Recall the notations related to symbol spaces, that were given in the introduction. We can consider the infinitely generated self-similar IFS where the mappings $f_{n,m}$ are from the construction of $C_\infty$. Define the projection mapping $\pi\colon\IN\to Q$ by 
\[
 \{\pi(\ii)\}=\bigcap_{n=1}^\infty f_{\iin{n}}(Q),
\]
where $f_{\iin{n}}=f_{i_1}\circ\dots\circ f_{i_n}$. Similarly, we denote $r_{\iin{n}}=r_{i_1}\cdots r_{i_n}$. It is  clear that the limit set $\bigcup_{\ii\in\IN}\{\pi(\ii)\}=:F$ is a subset of $C_\infty$. Since $C_\infty$ is compact, we also have that $\overline{F}\subset C_\infty$. On the other hand, let $x\in S$ where 
\begin{equation}
 \label{Sdef}
 S=\{x\in C_\infty: x=f_\ii(0)\text{ for some }\ii\in I^*\},
\end{equation}
and let $\jj\in\IN$ be arbitrary. Fix $x=f_\ii(0)$. It is clear from the construction of the mappings that $\pi((n,1)\jj)\to 0$ as $n\to\infty$. Thus we have $\pi(\ii (n,1)\jj)\to x$ and so $S\subset \overline{F}$. Since $S$ is dense in $C_\infty$ we also have that $C_\infty\subset\overline{F}$. Now we have that the set $C_\infty$ is the closure of a limit set of an infinitely generated self-similar IFS. This gives us tools for estimating the dimensions of $C_\infty$. By \cite[Corollary 3.17]{MauldinUrbanski1996}, we know that
\[
 \dimh F =\inf\{ s>0: \sum_{i\in I}r_i^s <1\}.
\]
This gives that $\dimh F >0$, by the reason that $\sum_{i\in I}r_i^s  > 1$ when $s$ is small. Since $F \subset C_\infty$, we have $\dimh C_\infty>0$.  Also, \cite[Theorem 3.1]{MauldinUrbanski1996} gives that $\dimp F=\dimp C_\infty = \ydimb F = \ydimb C_\infty$.

\subsection{Tangent properties of $C_\infty$}
Here we investigate the tangent sets of $C_\infty$. Theorem \ref{case1} shows that the tangent space of $C_\infty$ is large in a dense subset. For finitely generated self-similar sets, the tangent sets should are similar to the set itself \cite[Theorem and Remarks]{Bandt2001}. For infinitely generated sets we have a similar result for ``finitely generated points'', meaning points of the form $\pi\ii$ with $\ii\in \IN_N=\{(n,m)\in I:n\leq N\}^\N$, see Theorem \ref{case2}.

\begin{theorem}
 \label{case1}
 The set $S$, defined in $\eqref{Sdef}$, is dense in $C_\infty$ and $\Tan(C_\infty,x)=\KK_0^+$ for all $x\in S$.
\end{theorem}
\begin{proof}
 The first claim is trivial. For the second one, fix $x=f_\ii(0)$. The mappings $f_\ii$ are made of scaling $r_\ii$ and translation $f_\ii(0)$, so we have $T_{0,t}(C_\infty)=T_{f_\ii(0),r_\ii t}(f_\ii (C_\infty) )$. By the choice of $r_\ii$, we have that $f_\ii(Q)\cap f_\jj(Q)=\emptyset$ for all different $\ii$ and $\jj$ with $|\ii|=|\jj|$, implying that $C_\infty\cap Q(f_\ii(0),r_\ii t)= f_\ii(C_\infty)\cap Q(f_\ii(0),r_\ii t)$ for all $0<t<1$. Thus $T_{0,t}(C_\infty)= T_{f_\ii(0),r_\ii t}(f_\ii (C_\infty) ) =T_{f_\ii(0),r_\ii t}(C_\infty )$ and so $\Tan(C_\infty,x)=\Tan(C_0,0)=\KK_0^+$ by Theorem \ref{thmC00}.
\end{proof}

\begin{theorem}
 \label{case2}
 If $x=\pi\ii$, where $\ii\in\bigcup_{n=1}^\infty\IN_n$, then we have  $\Tan(C_\infty,x)\subset \{Q\cap(\alpha C_\infty+\beta)\}_{\alpha\geq 1, \beta\in[-\alpha,\alpha]}$.
\end{theorem}
\begin{proof}
 Let $\ii\in\IN_N$ and let $T_{x,t_n}(C_\infty)\hdarrow K$. Denote the smallest contraction ratio of the mappings $f_{n,m}, n=1,\dots N$ by $\underline{c}(N)$. For all $t>0$ let $k(t)$ be so that $\diam(\pi[\iin{k(t)}])\leq t < \diam(\pi[\iin{k(t)-1}])$. By our construction and choice of $\eps_n$, we have that
 \[
  \diam(\pi[\iin{n}])\leq\dist(\pi[\iin{n}],\pi[\jjn{n}])
 \]
 for all different $\ii,\jj\in\IN$ and $n\in\N$. Thus
 \[
  C_\infty\cap Q(x,t)=f_{\iin{k(t)-1}}(C_\infty) \cap Q(x,t) =\pi[\iin{k(t)-1}] \cap Q(x,t).
 \]
 Since $\diam(\pi[\iin{k(t)}])\leq t < \diam(\pi[\iin{k(t)-1}])\leq \diam(\pi[\iin{k(t)}])\underline{c}(N)^{-1}$ we now have, by self-similarity, that $T_{x,t_n}(C_\infty)=Q\cap (\alpha_n C_{\infty}+\beta_n)$, where  $1\leq\alpha_n\leq\underline{c}(N)^{-1}$ and $\beta_n\in[-\alpha_n,\alpha_n]^d$.
 
 Let $(j(n))_{n=1}^\infty$ be a sequence so that both $\alpha_{j(n)}$ and $\beta_{j(n)}$ converge. Call the limits $\alpha$ and $\beta$ respectively. Now it is clear that
 \[
  T_{x,t_{j(n)}}(C_\infty)=Q\cap (\alpha_{j(n)}C_\infty+\beta_{j(n)})
  \hdarrow Q\cap (\alpha C_\infty+\beta)
 \]
 Since $T_{x,t_n}(C_\infty)$ converges to $K$, the limit is the same for all sub-sequences and thus we have proved the claim.
\end{proof}
\begin{remark}
\label{caseremark}
 It is a bit technical to say exactly which collections of $\alpha$ and $\beta$ are needed to get $\Tan(C_\infty,x)= \{Q\cap(\alpha C_\infty+\beta)\}_{\alpha, \beta}$. If $x=f_{\ii}(0)$ and $|\ii|=N$, then $[x-t,x+t]^d$ intersects infinitely many cylinders of level $N$ for all $t>0$ and the above proof does not work. Also, if $\ii\in\IN\setminus\bigcup_{n=1}^\infty\IN_n$ then the sequence $\alpha_n$ is not bounded in general and so the above proof does not work.
\end{remark}

\section{Geometry and dimension}

In this section, we give a closer study to the geometric properties of locally rich sets in Euclidean spaces. We also study different dimensions of such sets.
\subsection{Tangent sets and dimension}
We show that using the construction of $C_\infty$, one can obtain a set of any given Hausdorff dimension that is locally rich in a dense subset. We begin by constructing a variant of $C_\infty$ so that it will have Hausdorff dimension zero. We start with the mappings $\{f_{n,m}\}$ that were used to create the set $C_1$. As has been showed, iterating these mappings gives a set with Hausdorff dimension strictly greater than zero. To make the Hausdorff dimension smaller, we only need to scale the mappings at each level. To this end, fix the sequence $(\alpha_k)_{k=1}^\infty$ with $\alpha_k=2^{-k}$, choose the sequence $(\eps_n)$ in the construction of the mappings $f_{n,m}$ so that $r_{n}\leq (2\#\gamma_n)^{-n}$, and for any $t>0$ let $k(t)\in\N$ be such that $1-tk(t)<-1$. For any $t>0$, we now have
\begin{align}
\label{finitesum}
 \sum_{\ii\in I} r_\ii^t
 &= \sum_{i=1}^{\infty}\# \gamma_i (r_i)^t \\
 &\leq \sum_{i=1}^{\infty} 2^{-it} (\# \gamma_i)^{1-it} \nonumber\\
 &\leq c'(t)+ \sum_{i=k(t)}^{\infty}2^{-it} \nonumber\\
 &= c(t)<\infty. \nonumber
\end{align}
Recall the notations from \eqref{ifs} - \eqref{symbol}. We define the collections of mappings $\FF_k$ by setting $\FF_k=\{\widetilde{f}_{n,m}: \widetilde{f}_{n,m}(x)= \alpha_k\lambda_n \widetilde{\delta}_{n} \eps_n x+\xi_{n,m} \}_{ (n,m)\in I} \cup \{f_{0,1}\}$ and mappings $\phi_k(A)= \cup_{\phi \in \FF_k}\phi(A)$ for $A \in \KK$ and $k\in\N$. Note that again $\phi_k(A)$ is compact for every $A \in \KK$ and $k\in \N$. Let 
\[
K_k:=\phi_1\circ\cdots \circ\phi_k(Q).
\]
Each $K_k$ is compact, and $K_k\subset K_{k-1}$ for all $k\geq 2$, so we get a non-empty compact set $ K_{\infty}=\bigcap_{k=0}^{\infty}K_k $. Note that $K_\infty$ is not a limit set of self similar IFS since we scaled the mappings at each level $k$ by a factor $\alpha_k$, and $\alpha_k\to 0$. By our choice of sequence $(\alpha_k)$ and $(\epsilon_k)$, and applying  Theorem \ref{thmC00}, we have that $K_\infty$ has the same tangent spaces as $C_\infty$ at the points of $\mathcal{S}$ defined in \eqref{Sdef}. It remains to show that $\dimh K_\infty=0$, as we claimed.
\begin{theorem}
 The set $K_\infty$ has Hausdorff dimension zero.
\end{theorem}
\begin{proof}
 It suffices to show that $\HH^t(K_\infty) = 0$ for all $t>0$. So fix $t>0$ and consider the covering of $K_\infty$ by $K_k$. Let $\delta(k)=\max\{r_\ii:\ii\in I^k\}$. Now using \eqref{finitesum} gives
 \begin{align*}
  \HH^{t}_{2 \sqrt{d} \delta(k)}(K_\infty)
  &\leq (2 \sqrt{d})^t \sum_{\ii\in I^k} (\alpha_\ii r_\ii)^t \\
  &= (2 \sqrt{d})^t (\prod_{i=1}^k \alpha_i)^t (\sum_{\ii\in I} r_\ii^t)^k\\
  &= (2 \sqrt{d})^t 2^{-\frac{1}{2}(k-1)(k)t} (c(t))^k \\
  &= (2 \sqrt{d})^t 2^{-\frac{1}{2}(k-1)(k)t+k\log_2 c(t)} \to 0
 \end{align*}
 as $k$ increases.
\end{proof}
\begin{corollary}
\label{Whitneys}
For every $s \in [0, d]$, there exist compact set $X$ with Hausdorff dimension $s$, and  countable dense subset $E \subset X$, such that $\Tan(X, x) =\KK_0^+$ for all $x\in E$.
\end{corollary}
\begin{proof}
It is well known that for every $s \in [0, d]$, there exist compact $F \in \KK_0^+$ with the Hausdorff dimension $s$ and without interior points, see \cite[Section 4.12]{Mattila1995}. We apply the Whitney's decomposition for the open set  $\mathbb{R}^d\setminus F$. There exists cubes $\{ Q_n \}$ such that $\mathbb{R}^d\setminus F = \bigcup^{\infty}_{n= 1} Q_n $, the interior of $Q_i$ and interior of $Q_j$ are disjoint for $i\neq j$, and
\begin{align*}
\diam (Q_n) \leq \dist(Q_n, F)\leq 4\diam(Q_n)
\end{align*} 
for all $n\in \N$. Denote the center point of $Q_n$ by $x_n$ and set $C=\bigcup^{\infty}_{n=1} \{x_n\}$. Choose a large closed ball $B$, such that it contains $F$, and let
$E = B \cap C$. It is clear that 
$F\subset \overline{E}$.
By setting 
\begin{align*}
X= F\cup \bigcup_{x\in C} (x+\frac{1}{5} \diam(Q_x)K_{\infty}).
\end{align*}
we have that $X$ and $E$ satisfy the given conditions.
\end{proof}
Like we mentioned in the end of Section \ref{locricmet}, even the fact that one can tangent out $\{0\}$ from all points of a set $E$, does not give any information on the Hausdorff dimension of $E$. See Theorem \ref{mattilacorollary} below.

\begin{theorem}
\label{mattilacorollary}
For any $s \in [0,d]$, there is a compact set so that it has Hausdorff dimension $s$ and one can tangent out $\{0\}$ at all its points.
\end{theorem}

For the proof, we briefly give the construction of homogeneous Cantor sets. For more details, see e.g. \cite[Chapter 4]{Mattila1995}, \cite[Chapter 4]{Falconer1990} or \cite{FengRaoWu1997}. Let $\{m_k\}^{\infty}_{k=1} \subset \N $ with $m_k\geq 2$ for all $ k\in \N$ and $\{\lambda_k\}^{\infty}_{k=1} \subset (0,1)$ with the property $m_k \lambda_k <1$ for all $k\in \N$. 
Let $E_0$ be the unit interval $[0,1]$.  For interval $[0,1]$ the  $1$ th level intervals $I_1, \cdots , I_{m_1}$
contained in $[0,1]$ are of equal length $\lambda_1$ and equally spaced 
with the left-hand ends of $I_1$ and $[0,1]$ coinciding, and the right-hand ends of $I_{m_1}$ 
and $[0,1]$ coinciding.  Given $E_k$, a collection of $\prod^{k}_{i=1} m_i$ disjoint interval with equal length $\prod^{k}_{i=1} \lambda_i$. For each $k$ th level interval $I$, the $(k + 1)$ th level intervals $I_1, . . . , I_{m_{k+1}}$
contained in $I$ are of equal length $\prod^{k+1}_{i=1} \lambda_i$ and equally spaced, 
with the left-hand ends of $I_1$ and $I$ coinciding, and the right-hand ends of $I_{m+1}$ and $I$ coinciding. We define the limit set of this construction by 
\[
E(\{m_k\}, \{\lambda_k\}) = \bigcap^{\infty}_{k=1} E_k.
\]
For the Hausdorff dimension of $E(\{m_k\}, \{\lambda_k\})$, we have that
\begin{equation}\label{s}
\dimh E(\{m_k\}, \{\lambda_k\}) =\liminf_{k\rightarrow\infty} \frac{\log\prod^{k}_{i=1} m_i }{-\log\prod^{k}_{i=1} \lambda_i},
\end{equation}
see \cite[Theorem 2]{FengRaoWu1997}.
\begin{proof}[Proof of Theorem \ref{mattilacorollary}]
For the case $s =0$,  we chose any finite union of points as our $E$. Let us continue in one dimension.
Let $0<s<1$, $m_k =k$ and  $\lambda_k = k^{-\frac{1}{s}}$ for  $k \in \N.$ By \eqref{s}, we have that $\dimh E(\{m_k\}, \{\lambda_k\})=s$.  
By a geometrical observation, for large $k$ we have
\begin{equation}\label{e}
\frac{\prod^{k}_{i=1} \lambda_i}{\dist(I,J)} \leq \frac{k-1}{k^{\frac{1}{s}}-k} 
\end{equation}
for all different $k$ th interval $I$ and $J$.
Since $0<s<1$, we have that the right hand side of \eqref{e} goes to zero as $k\rightarrow\infty$. This implies that $\{0\} \in \Tan(E,x)$ for all $x\in E:=E(\{m_k\}, \{\lambda_k\})$.

Next, let $E=\prod^{d}_{i=1} E_i$, where $E_i=E(\{m_k\}, \{\lambda_k\})$ for all $i$.
For any point $x \in E$, there are $x_i \in E_i$, such that $x=(x_1,\cdots, x_d)$. 
Let $\epsilon >0$. By the structure of $E( \{m_k\}, \{\lambda_k\})$, there is positive $t_n$, such that $d_H (T_{x_i, t_n}(E_i), \{0\}) < \epsilon$ for all $i$. Thus  $d_H (T_{x, t_n}(E), \{0\}) < \epsilon \sqrt{d}$. By the arbitrary choice of $\epsilon$, we have that $\{0\}\in \Tan(E,x)$ for all $x\in E$. 
By a general product formula \cite[Product formula 7.2]{Falconer1990}, the number $sd$ is a lower bound for $\dimh E$. By using the natural covering of $E$, we get
\begin{align*}
 \HH^{ds}(E)
 & \leq \liminf_{n\to\infty} \sum_{i=1}^{\left( \prod_{k=1}^n m_k \right)^d} |\sqrt{d}\prod_{k=1}^n \lambda_k|^{ds} \\
 & = d^{\frac{ds}{2}} \liminf_{n\to\infty} \left( \prod_{k=1}^n m_k \right)^d |\prod_{k=1}^n \lambda_k|^{ds} \\
 & = d^{\frac{ds}{2}} \left( \liminf_{n\to\infty} n! ((n!)^{-\frac{1}{s}})^s \right)^d \\
 & = d^{\frac{ds}{2}},
\end{align*}
which gives the upper bound $\dimh E \leq sd$.

For the case $s=d$, let $0 <s_1<s_2 \cdots$ be an increasing sequence numbers with  $s_n \nearrow d$. By the above argument, we have that  for each $s_n$,  there is a set $E_n$ with $\dim_H E_n = s_n$ and $\{0\}$ is a tangent set at all points of each $E_n$. Let 
\[
E= \{0\} \cup \bigcup^\infty_{n=1} (2^{-n^{2}}+n 2^{-n^{2}}E_n).
\]
We have that  $\dim_H E = \sup_{n\geq1} \{ \dim_H E_n\}=d$. For the point zero, applying the same argument as in Theorem \ref{thmC00}, gives that  $\{0\}\in \Tan(E,0)$. For other points $x \in E$, this follows by our choice of $E_n$, and the sequences $2^{-n^{2}}$ and $n2^{-n^{2}}$. Thus we have completed the proof.
\end{proof}

\begin{example}\label{density}
Let $E \in \KK$ and $x$ be a Lebesgue density point of $E$, then $\Tan(E, x) = Q$.
\begin{proof}
Denote $r_E = \sup \{ r : B(x,r) \subset Q\setminus E\}$, by elementary geometry we see that $d_H( E, Q) = r_E$. Denote by $\mathcal{L}$ the Lebesgue measure on $\R^d$. It's not hard to see that $\frac{\mathcal{L}(E\cap B(x,r))}{\alpha(d)r^d} \rightarrow 1$ if and only if $\frac{\mathcal{L}(E\cap Q(x,r))}{(2r)^d} \rightarrow 1$ where $\alpha(d)$ is the measure of the unit ball and $Q(x,r)$ means the cube with center $x$ and side-length $2r$.

Let $t_n \searrow 0$.  Set $E_n := T_{x, t_n}(E)$, we have $\mathcal{L}(E_n) = \mathcal{L}(E\cap Q(x, t_n))t_n ^{-d}$. If $x$ is the density point of $E$, then $\mathcal{L}(E_n) \rightarrow 2^d=\mathcal{L}(Q)$. It means that $r_{E_n} \rightarrow 0$. By the above result, we conclude that $E_n \hdarrow Q.$ By the arbitrary choice of $t_n$, we have completed the proof.
\end{proof}
\end{example}

Notice that  the converse is not true in the above example. For example, let $E= \{0\}\cup\{\pm \frac{1}{n}\}_{n\geq 1}$, then $\Tan(E,0)= [-1, 1]$.

\begin{example}
Let $E\in \KK$ with $0<\HH^s(E)<\infty$. If $\{0\} \in \Tan(E,x)$ for  $\HH^s$ almost all $x\in E$, then $\underline{D}^s(E,x)=0$ for $\HH^s$ almost all $x \in E$. Here $\underline{D}^s(E,x)$ and $\overline{D}^s(E,x)$ mean the lower and upper $s$-densities of $E$ at $x$ respectively, see \cite[Chapter 6]{Mattila1995}.

Since $ \{0 \} \in \Tan(E,x)$, there exist a sequence $(r_n)$ decreasing to $0$, such that $d_H(T_{x,r_n}(E,x), \{0\})\rightarrow 0$. Thus for any $\epsilon>0$, there exist $N$, such that 
$d_H(T_{x,r_n}(E,x), \{0\})< \epsilon$ for all $n\geq N$. It means that $(B(x, r_n) \setminus B (x, \epsilon r_n))\cap E =\emptyset$, then
\[
 \frac{\mathcal{H}^s(E\cap B(x,r_n))}{ 2^{-s} r_n^s}
 =\frac{\mathcal{H}^s( E\cap B(x, \epsilon r_n) )}{ 2^{-s} (\epsilon r_n)^s}\epsilon^s.
\]
Let $n \rightarrow \infty$, then $\underline{D}^s(E,x)\leq \epsilon^{s}\overline{D}^s(E,x)$. Since $\overline{D}^s(E,x)\leq 1$ for $\mathcal{H}^s$ a.e. $x \in E$, the results follows. 
\end{example}

\subsection{Porosity of sets}
A subset $E$ of $\R^d$ is called porous,  if there exist $0<\alpha<1$, such that for each $x\in \R^d$, and $r>0$, the ball $B(x,r)$ contains a open ball $U(y, \alpha r)$ that does not meet $E$ (, meaning that $E \cap U(y, \alpha r) = \emptyset$).  In this case we also call $E$ $(\alpha)$-porous.  The concept of porosity is closely related to the Assouad dimension. The connection is the following: A subset $E$  of $ \mathbb{R}^d$ is porous if and only if $\dim_A E <d$.  For details, we refer to \cite[Theorem 5.2]{Luukkainen1998}.  

\begin{proposition}\label{porous}
Let $E$ be $(\alpha)$-  porous, then  every element of $\Tan(E,x)$ is also $(\alpha)$- porous for all $x \in E$.
\end{proposition}
\begin{proof}
Let $x \in E$ and $F \in \Tan(E,x)$.  Then there exist a sequence $t_n \searrow 0$, such that $d_H(T_{x,t_n}(E), F) \rightarrow 0$. Denote $E_n:=T_{x,t_n}(E) , n \in \mathbb{N}$. Let $B(y,r)$ be a ball of $\R^d$. Then for every $\epsilon \in (0, \alpha r)$ there exists $N(\epsilon)$, such that  
\begin{equation}\label{key}
d_H( E_{N(\epsilon)}, F) < \epsilon. 
\end{equation}
By the porosity  of $E$ and a simple geometric observation, we have that $E_{N(\epsilon)}$ is $(\alpha)$- porous. Thus there is $z_{N(\epsilon)}\in\R^d$ so that $U(z_{N(\epsilon)}, \alpha r) \subset B(y,r)$ and $U(z_{N(\epsilon)}, \alpha r) \cap E_{N(\epsilon)} = \emptyset$. Together with \eqref{key} we have
\begin{equation}\label{haha}
U(z_{N(\epsilon)}, \alpha r-\epsilon ) \cap F = \emptyset. 
\end{equation}

Let $\epsilon \rightarrow 0$, then there is $z_{N(\epsilon)}$ such that \eqref{haha} holds. Note that  $z_{N(\epsilon)} \in B(y,r)$ for all $\epsilon \in (0, \alpha r)$. So there is a converging sub-sequence of $z_{N(\epsilon)}$ denoted by $z_{N(\epsilon_j)}$ with $\epsilon_j \searrow 0$ such that $z_{N(\epsilon_j)}\rightarrow z$. Applying  $\eqref{haha}$ to every $z_{N(\epsilon_j)}$, we conclude that $U(z, \alpha r) \cap F = \emptyset$. By the arbitrary choice of 
ball $B(y,r)$ and $x \in E$, we have completed the proof.
\end{proof}

By applying Proposition \ref{porous}, we know that a locally rich set can't be a porous set, and thus it has Assouad dimension $d$ (Here we apply the fact that:  A subset $E$ of $\mathbb{R}^d$ is porous if and only if $\dim_A E < d$ again).  We put this result as the following corollary. 

\begin{corollary}\label{locallyrichAd}
 Let $E \in \KK$ be a  locally rich set, then $\dim_A E = d$.
\end{corollary}

\begin{definition} 
Let $E \subset \R^d$ and $x \in \R^d$. The local porosity of $E$ at $x$ at distance $r$ is 
\[
\por(E,x,r)=\sup \{ \alpha \geq 0: 
B(y, \alpha r) \subset B(x,r) \backslash E ~~\text{for some}~~ y \in \R^d\}.
\]
The upper and lower porosities of $E$ at $x$ are defined as 
\[
\overline{\por}(E,x) =\limsup_{r\rightarrow 0} \por(E,x,r) ~~\text{and }~~ \underline{\por}(E,x)=\liminf_{r\rightarrow 0} \por(E,x,r),
\]
respectively.
\end{definition}
For the background and more applications of porosity, we refer to \cite{Shmerkin2011survey}. By the definition, we have 
\[
0\leq \underline{\por}(E,x) \leq \overline{\por}(E,x) \leq \frac{1}{2},
\]
for any $x \in E$. Note that if $E$ is a $(\alpha)$-porous set of $\R^d$, then $\por(E,x,r) \geq \alpha$ for any $x \in \R^d$ and $r >0$.  Thus $\inf_{x\in E} \underline{\por}(E,x)\geq \alpha.$ We conclude that if $E$ is a porous set of $\R^d$, then $\inf_{x\in E} \underline{\por}(E,x)> 0$. The following simple example shows that the converse is not true.

\begin{example} Let $E = \{0\}\cup \{\frac{1}{n}\}_{n=1}^\infty$ be a subset of $\R$,  then 
\[\underline{\por}(E,x)=\overline{\por}(E,x)=\frac{1}{2} ~~\text{for all}  ~~x \in E,
\]
and $E$ is not porous in $\R$.
\end{example}

\begin{proposition}\label{d}
Let $E \in \KK$. If $\{0\} \in \Tan(E,x)$, then $\overline{\por}(E,x)=\frac{1}{2}$.
\end{proposition}
\begin{proof}
There exist $t_n \searrow 0$, such that  $d_H(T_{x,t_n}(E), \{0\}) \rightarrow 0$ as $n \rightarrow\infty$. For any $\epsilon \in (0, 1)$, there is $N$ such that for any $n \geq N$, we have 
\begin{equation}
d_H(T_{x,t_n}(E), \{0\})<\epsilon.
\end{equation}
Thus there is an open ball $U(z, \frac{1-\epsilon}{2}) \subset B(0,1)$ and $U(z, \frac{1-\epsilon}{2})\cap T_{x,t_n} (E) = \emptyset$. 
We also have $T_{x,t_n}^{-1}(U(z,\frac{1-\epsilon}{2})) \cap E = \emptyset$. It follows that $\por(E,x,t_n)\geq \frac{1-\epsilon}{2}.$
Let $\epsilon\rightarrow 0$, then $\overline{\por}(E,x)\geq\frac{1}{2}$. We complete the proof, since  the natural upper bound of $\overline{\por}(E,x)$ is $\frac{1}{2}$.
\end{proof}
As an implication of Theorem \ref{Baire} and Proposition \ref{d}, we have the following corollary. 

\begin{corollary} 
Typically a compact set of $\KK$ has upper porosity $\frac{1}{2}$  at all its points.
\end{corollary}

A Borel regular measure $\mu$ on metric space $X$ is said to be \emph{doubling} if there is a constant $C \geq 1$ such that for any $x \in X$ and
$0<r<\infty$ we have
\begin{equation}
\label{doubling}
0 < \mu(B(x,2 r)\leq C\mu(B(x,r))<\infty.
\end{equation}
In this case we also say that $\mu$ is $C$-doubling. We denote by
$\mathcal{D}(X)$ the collection of non-trivial doubling measures on $X$. Let $E \subset X$, we say that $E$ is \emph{thin} if $E$ has zero measure for all doubling measures on $X$. It is well known that porous sets are thin, but the converse is not true, for more details see \cite{Wu1993}. In the following, we only consider $\mathcal{D}([0,1]^d)$.

It's not hard to see that by applying the density theorem and Proposition \ref{d} we have the following fact. For the convenience of reader, we  show the details here.

\begin{proposition}\label{thin}
Let $E \in \KK$ and $ \{0\} \in \Tan(E,x)$ for all $x \in E$, then $E$ is thin.
\end{proposition}
\begin{proof}
Let $x \in E$. By proposition $\ref{d}$, we have $\overline{\por}(E,x)= \frac{1}{2}$. There exist a sequence $(t_n)$, $t_n \searrow 0$, such that $\lim_{n\rightarrow\infty}\por(E,x,t_n)\rightarrow \frac{1}{2}$. Thus for every $\epsilon > 0$, there exists an integer 
$N$, such that  for all $n\geq N$ we have  $E \cap (B(x, t_n)\setminus B(x, \epsilon t_n))= \emptyset$. 

Let $\mu \in \mathcal{D}(Q)$. By the above calculation, it is easy to see that for all $n\geq N$, 
\begin{equation}
\dfrac{\mu(E \cap B(x, t_n))}{\mu(B(x,t_n))}
\leq \dfrac{\mu(B(x, \epsilon t_n))}{\mu (B(x,t_n))} \leq C \epsilon^s,
\end{equation}
where positive constants $C$ and $s$ only depend on the measure $\mu$, see \cite[Chapter 13]{Heinonen2001}. By the above estimate we have that $\liminf_{r\rightarrow 0} \dfrac{\mu(E \cap B(x, r))}{\mu(B(x,r))} =0$ for all $x \in E$. Thus we have $\mu(E) = 0$ by applying the density point theorem for Radon measures, see \cite[Corollary 2.14]{Mattila1995}.
\end{proof}

By applying Proposition \ref{thin}, we have the following corollary.
\begin{corollary}\label{locallyrichisthin}
Let $E \in \KK$ be a locally rich set, then $E$ is thin.
\end{corollary}

\begin{corollary}\label{known}
By Theorem \ref{Baire} and Corollaries \ref{locallyrichAd} and \ref{locallyrichisthin}, we conclude that a typical set of $\KK$ has Assouad dimension $d$ and is thin.
\end{corollary}

\begin{remark}\label{Box}
Corollary \ref{known} follows implicitly from \cite{FengWu1997} in a different manner. In \cite{FengWu1997}, they proved that a typical set $E \in \KK$ has $\adimm (E)=0$ and $\ydimm(E)=d$. By the well known results that $\dim_H (E) \leq \adimm(E) \leq \ydimm (E) \leq \dim_A (E)$, we have the dimension part of our claim. For the thin part, we recall the following results: A subset of uniformly perfect metric space with Hausdorff dimension zero is thin \cite[Chapter 13]{Heinonen2001}.
\end{remark}

By Theorem \ref{Baire} and Remark \ref{Box}, we know that typically a compact set $X \subset \R^d$  is locally rich and has $\overline{\dim}_M X=d$. We end this section with the following question. 
\begin{question}
If $X \subset \R^d$ is locally rich, is it true that $\overline{\dim}_M X=d ?$
\end{question}

\section{Global structure of sets}

The converse to the local structure of sets is the global structure of sets. Here we briefly discuss, what can be the global structure of an unbounded closed set in Euclidean space. In other words, what can we see when zooming out? To be formal, we show the following definitions. We say that $F \in\KK$ is a photograph of a closed $E\subset\R^d$ at $x\in E$ if there exists a sequence $t_n\nearrow\infty$ so that
 \[
  T_{x,t_n} (E)\hdarrow F.
 \]
Denote  by $P(E,x)$  all the photographs of $E$ at $x$. We say that $E$ is \emph{globally rich} at $x\in E$ if $P(E,x)\approx\KK$. If this happens for all $x\in E$ then we just say that $E$ is \emph{globally rich}. Note that the similar concept in metric space was called Gromov-Hausdorff asymptotic cone, see \cite[p-276]{BugaroBugaroIvanov2001}.

Recall the subspace $\KK_0^+$ of $\KK$ and the choice of $(\gamma_i)_{i=1}^{\infty}$ at the beginning of Section \ref{Euclidean}. Fix sequences $a_n \nearrow \infty$ and $\lambda_n \nearrow \infty$ with the following properties
\begin{enumerate}
 \item $a_{n}+\lambda_{n}+1\leq a_{n+1}$ for all $n\in\N$
 \item $\frac{a_n}{a_n+\lambda_n} \leq \frac{1}{n}$ as $n\to\infty$
\end{enumerate}
For example, one could choose $a_n=2^{n^2}$ and $\lambda_n=n a_n$ for all $n\in\N$. We set $\widetilde{\gamma}_n= \vec{a}_n +\lambda_n\gamma_n$, where $\vec{a}_n=(a_n,0,0,\ldots)$ and let
\[
A:=\bigcup_{n=1}^\infty \widetilde{\gamma}_n.
\]
The set $A$ is made of the elements of $(\gamma_n)_{n=1}^{\infty}$ scaled with $\lambda_n$ and positioned so that the leftmost point of the $n$ th piece on the $x_1$-axis is $a_n$. The condition $(2)$ gives the separation of the pieces. Obviously $A$ is closed set of $\R^d$. By applying the same argument as in Theorem
\ref{thmC00} to this zooming out case, we have the following result. 
We omit the proof here, but note the result is the same for all points, since any two points get arbitrarily close when looking from far enough. This is a major difference to the zooming in case.
\begin{theorem}\label{photo}
 We have $P(A, x)= \KK_0^+$ for all $x \in A$.
\end{theorem}

Note that the set $A$ of Theorem \ref{photo} is a countable set of (different) points. By our earlier results we know that there exist locally rich sets. Now we are going to put the copy of a locally rich set of suitable size at every point of $A$, thus we can have one set such that it is not only locally rich, but also globally rich. 

\begin{corollary}
There exists a closed set  $ E \subset \R^d$ such that $P(E,x) \approx \KK$ and $Tan(E,x) \approx \KK$ for all $x \in E$.
\end{corollary}
\begin{proof}
Let $A$ be the set of Theorem \ref{photo} and $R \subset Q$ a locally rich set with $0\in R$ and diameter less than one. Let
\[
E = \bigcup^\infty_{n=1} \bigcup_{x \in \widetilde{\gamma}_n} (x+2^{-n}\delta_nR),
\]
where $\delta_n$ is from \eqref{delta}. 
It's follows directly that $E$ is a closed set. For any $x \in E$, by the choice of $2^{-n}$ we still have that $P(E, x)= \KK_0^+$. By the definition of $(\widetilde{\gamma}_n)^
{\infty}_{n=1}$ and $(\delta_n)^
{\infty}_{n=1}$, we have that $( a +2^{-m}\delta_m R) \cap ( b +2^{-n}\delta_n R) = \emptyset$  for $a \in\widetilde{\gamma}_m, b \in \widetilde{\gamma}_n$  with $a \neq b$. Thus by the local richness of $R$ we have $\Tan(E,x) \approx \KK$. We complete the proof by the simple fact that $\KK_0^+\approx \KK$.  
\end{proof}

\subsection*{Acknowledgements}
The authors would like to thank Antti K{\"a}enm{\"a}ki for reading the manuscript and giving useful comments, and Tuomo Ojala and Ville Suomala for many helpful discussions. Especially we would like to thank the anonymous referee for carefully reading the manuscript and giving exellent comments, and thus improving the quality of this article.

\bibliographystyle{amsplain}
\bibliography{Bibliography.bib}
\end{document}